\documentclass[12pt,centertags,oneside,reqno]{amsart}
\usepackage{amsmath,amsthm,amssymb,bm}

\usepackage{amscd,amsxtra,calc}
\usepackage{cmmib57}
\usepackage{url}

\usepackage{amscd}
\usepackage{pstricks}
\usepackage{color}
\usepackage[a4paper,width=16.2cm,top=3cm,bottom=3cm]{geometry}

\numberwithin{equation}{section}

\theoremstyle{plain}

\newtheorem{thm}{Theorem}[section]
\newtheorem{theorem}[thm]{Theorem}
\newtheorem{thma}{Theorem}

\newtheorem{proposition}[thm]{Proposition}

\theoremstyle{definition}

\newtheorem{remark}[thm]{Remark}

\newtheorem{definition}[thm]{Definition}

\newtheorem{example}[thm]{Example}

\newtheorem{question}[thm]{Question}

\numberwithin{equation}{section}

\newtheorem{report}{report}
\newtheorem{exercise}{excercise}[section]

\usepackage{mleftright}
\usepackage{mdwlist}

\usepackage{etoolbox}

\makeatletter
\def\subsection{\@startsection{subsection}{1}%
  \z@{.5\linespacing\@plus.7\linespacing}{-.5em}%
  {\normalfont\itshape}}
\patchcmd{\@sect}
  {\@addpunct.}
  {}
  {}
  {}
\makeatother

\allowdisplaybreaks

\makeatletter
\def\subsection{\@startsection{subsection}{2}%
  \z@{.5\linespacing\@plus.7\linespacing}{.3\linespacing}%
  {\normalfont\bfseries}}
\makeatother

\usepackage{listings}
\usepackage{caption}

\usepackage{amsmath}

\usepackage{mathtools}

\newcommand{\id}{\mathrm{id}}

\usepackage[pdftex]{graphicx}

\usepackage{threeparttable}

\usepackage{appendix}

\usepackage{comment}

\title[]{A birational map of a projective space whose intermediate dynamical degrees are all transcendental}

\author{Yutaro Sugimoto
}
\begin{document}
\maketitle

\begin{abstract}
 We construct a birational map of $\mathbb{P}^d$ ($d\geq6$) whose intermediate dynamical degrees are all trancendental.
\end{abstract}

\section{Introduction}
 Let $X$ be a smooth projective variety over an algebraically closed field $k$ with characteristic $0$.
 For a dominant rational map $f:X\dashrightarrow X$ and $0\leq p\leq \mathrm{dim}(X)$, the $p$-th dynamical degree of $f$ is defined by
 \begin{align*}
  \lambda_p(f)=\lim_{n\to\infty}((f^n)^{*}(H^p)\cdot H^{\mathrm{dim}X-p})^{\frac{1}{n}},
 \end{align*}
 where $H$ is an ample divisor (cf.\ \cite{Tru15}).\par
 The pullback map $f^{*}$ and the pushforward map $f_{*}$ is defined as follows.
 For a birational map $f:X\dashrightarrow X$, let $\Gamma_f$ be the closure of the graph of $f$.
 Pullback $f^{*}:A_{l}(X)\rightarrow A_{l}(X)$ is defined as $f^{*}(\gamma):={\pi}_{*}([\Gamma_f]\cdot\pi^{*}\gamma)$ and $f_{*}(\gamma):={\pi}_{*}([\Gamma_f]\cdot\pi^{*}\gamma)$ (cf.\ \cite[Definition 16.1.2]{Ful98}), where $\alpha$ is an algebraic cycle $\gamma$ on $X$.
 \begin{figure}[h]
  \includegraphics[scale=0.3]{graph.jpg}
  \label{figure}
 \end{figure}\par
 Also,
 \begin{align*}
  \lambda_p(f)=\lim_{n\to\infty}\|(f^n)^{*}:N^p(X)\otimes_\mathbb{Z}\mathbb{R}\rightarrow N^p(X)\otimes_\mathbb{Z}\mathbb{R}\|^{\frac{1}{n}},
 \end{align*}
 where $N^p(X)$ is the group of algebraic cycles of codimension $p$ on $X$ modulo numerical equivalence (cf.\ \cite{Tru15}, \cite{Tru20}).\par
 When $k=\mathbb{C}$, $X$ is considered as a compact K\"{a}hler manifold, $f$ is considered as a meromorphic map, and the dynamical degrees can be calculated as
 \begin{align*}
  \lambda_p(f)&=\lim_{n\to+\infty}\left(\int_{X} (f^n)^{*}(\omega_X^p)\wedge\omega_X^{\mathrm{dim}X-p} \right)^{\frac{1}{n}}\\
              &=\lim_{n\to+\infty}\|(f^n)^{*}:\mathrm{H}^{p,p}(X)\rightarrow\mathrm{H}^{p,p}(X)\|^{\frac{1}{n}}
 \end{align*}
 where $\omega_X$ is a normalized K\"{a}hler form, $\mathrm{H}^{p,p}(X)$ ($0\leq p\leq \mathrm{dim}X$) is the $(p,p)$-Dolbeault cohomology, and $\|\cdot\|$ is a norm for linear transformations (cf.\ \cite{DS04}, \cite{DS05}, \cite{DS17}).\par
 The value of the dynamical degrees are not less than $1$ and $\lambda_0(f)=1$.
 When $f$ is a morphism, the dynamical degrees of $f$ are algebraic, since $(f^n)^{*}=(f^{*})^n$.
 For the case $X=\mathbb{P}^d$, the rational map $f:\mathbb{P}^d\dashrightarrow\mathbb{P}^d$ can be written as $f=[f_0:f_1:\cdots:f_d]$ by the homogeneous polynomials (over $k$) of the same degree.
 Define $\mathrm{deg}(f):=\mathrm{deg}(f_0)$ and then the first dynamical degree of $f$ is calculated by
 \begin{align*}
  \lambda_1(f)&=\lim_{n\to\infty}((f^n)^{*}(H)\cdot H^{d-1})^{\frac{1}{n}}\\
              &=\lim_{n\to\infty}(\mathrm{deg}(f^n))^{\frac{1}{n}}
 \end{align*}\par
 When $d=2$, it is known that the first dynamical degree of $f$ is always algebraic by the next theorem.
 \begin{theorem}[{\cite[Theorem 5.1(1)]{DF01}}]
  Let $X$ be a compact K\"{a}hler surface and $f\colon X\dashrightarrow X$ be a bimeromorphic map with $\rho(f^{*})>1$ for the operator $f^{*}$ on $\mathrm{H}^{1,1}(X)$ where $\rho$ is the spectral radius.\par
  Then the operator $f^{*}$ has exactly one eigenvalue $\lambda\in\mathbb{R}_{>0}$, of modulus $|\lambda|>1$, and in fact $\lambda=\rho(f^{*})$.
 \end{theorem}
 This implies that $\lambda_1(f)$ is $1$ or a Pisot number or a Salem number (see definition in \cite{BDGPS92}).
 Realizable Salem numbers as $\lambda_1(f)$ for some birational map $f:\mathbb{P}^2\dashrightarrow\mathbb{P}^2$, are all determined in \cite[Theorem 1.1]{Ueh16}.
 Realizable Pisot numbers as $\lambda_1(f)$ are partially determined in \cite[Theorem C]{Kim23} (cf.\ \cite[Section 3]{BK06}, \cite[Section 7]{BC16}).\par
 In \cite{BDJK24}, a birational map of $\mathbb{P}^d$($d\geq3$) with transcendental first dynamical degree of $f$, is constructed.
 \begin{thm}[{cf.\ \cite[Theorem 1.1]{BDJK24}}]\label{BDJK theorem}
  For each $d\geq3$, there exists a bitrational map $f:\mathbb{P}^d\dashrightarrow\mathbb{P}^d$, which has trancsendental (first) dynamical degree $\lambda_1(f)$.
 \end{thm}
 In \cite[Section 7.2]{BDJK24}, the specific birational map $\phi:\mathbb{P}^3\dashrightarrow\mathbb{P}^3$ with transcendental (first) dynamical degree is presented.\par
 \begin{remark}
  The constructed $f$ in Theorem \ref{BDJK theorem} is of the form $f=L_{B^{-1}}\circ h_{-I}\circ L_{B}\circ h_A$, where
  \begin{itemize}
   \item $h_{M}:\mathbb{P}^d\dashrightarrow\mathbb{P}^d$ is the monomial map corresponds to $M\in\mathrm{GL}_{d}(\mathbb{Z})$
   \item $L_{M'}:\mathbb{P}^d\dashrightarrow\mathbb{P}^d$ is the linear transformation corresponds to $M'\in\mathrm{GL}_{d+1}(k)$
  \end{itemize}
  and
  \begin{align*}
   I=\begin{pmatrix}
      1 & 0 & 0 & \cdots & 0 \\
      0 & 1 & 0 & \cdots & 0 \\
      0 & 0 & 1 & \cdots & 0 \\
      \vdots & \vdots & \vdots  & \ddots & \vdots \\
      0 & 0 & 0 & \cdots & 1
     \end{pmatrix},\ 
   B=\begin{pmatrix}
      1 & -1 & 1 & \cdots & (-1)^{d} \\
      1 & 1 & -1 & \cdots & (-1)^{d-1} \\
      -1 & 1 & 1 & \cdots & (-1)^{d-2} \\
      \vdots & \vdots & \vdots  & \ddots & \vdots \\
      (-1)^{d-1} & (-1)^{d-2} & (-1)^{d-3} & \cdots & 1
     \end{pmatrix}.
  \end{align*}
  Also, the specific example in \cite[Section 7.2]{BDJK24} is the case that $d=3$ and
  \begin{align*}
   A=\begin{pmatrix}
      -3 & -14 & -12 \\ 4 & 11 & 6 \\ -2 & -4 & -1
     \end{pmatrix}.
  \end{align*}
 \end{remark}
 In \cite[Section 3]{Sug24}, we calculate the second dynamical degree of the map $\phi$ and confirm $\lambda_1(\phi)<\lambda_2(\phi)$.
 Also, we construct another birational map $\psi:\mathbb{P}^3\dashrightarrow\mathbb{P}^3$ with transcendental first dynamical degree  with $\lambda_1(\psi)>\lambda_2(\psi)$ (cf.\ \cite[Section 4]{Sug24}) by substituting
 \begin{align*}
  A=\begin{pmatrix}
     56 & -19 & -17 \\ -16 & 6 & 5 \\ 207 & -71 & -63
    \end{pmatrix}.
 \end{align*}
 It is confirmed that $\lambda_2(\psi)$ is algebraic as below.
 \begin{thm}[{cf.\ \cite[Section 4]{Sug24}}]\label{1-cohomologically hyperboic}
  There exists a birational map $\psi:\mathbb{P}^3\dashrightarrow\mathbb{P}^3$ whose dynamical degrees are
  \begin{align*}
   \lambda_0(\psi)=1, \lambda_1(\psi)=\lambda, \lambda_2(\psi)=a, \lambda_3(\psi)=1,
  \end{align*}
  with $291\leq\lambda\leq669$ is a transcendental number and $a=174.6660\cdots$ is a root of 
  \begin{align*}
   a^9-173a^8-291a^7-2a^6+332a^5+334a^4+238a^3+75a+75=0.
  \end{align*}
 \end{thm}
 We use this result in obtaining a transcendental arithmetic degree in \cite[Corollary 4.5]{Sug24} by combining with the results of \cite[Theorem B]{MW24}.\par
 In this paper, we use Theorem \ref{1-cohomologically hyperboic} to resolve the next question by using the algebraicity of $a$.
 \begin{question}[{\cite[Question 1.6]{BDJK24}}]
  Does there exist a birational map $f:\mathbb{P}^d\dashrightarrow\mathbb{P}^d$ for which all intermediate dynamical degrees are transcendental?
 \end{question}
 Now the word \textit{intermediate dynamical degrees} means $\lambda_1(f),\lambda_2(f),\ldots,\lambda_{d-1}(f)$.\par
 It is known that the intermediate dynamical degrees of an automorphism $f$ of a compact K\"{a}hler manifold (or a projective variety over $\mathbb{C}$) are either irrational or equal to $1$ (\cite[Theorem 6.1]{Bed11}). 
 \begin{thma}\label{Main theorem}
  For $d\geq6$, there exists a birational map $f:\mathbb{P}^d\dashrightarrow\mathbb{P}^d$ (over an algebraic closed field $k$ with characteristic $0$), whose intermediate dynamical degrees are all transcendental.
 \end{thma}
 This theorem is proved in Section \ref{proof of theorem a}.\par
 \cite[Question 1.6]{BDJK24} presents another question about dynamical degrees.
 \begin{question}[{\cite[Question 1.6]{BDJK24}}]
  Does there exist a birational map $f:\mathbb{P}^d\dashrightarrow\mathbb{P}^d$ for which $\lambda_1(f)$ is transcendental and also maximal among intermediate dynamical degrees $\lambda_i(f)$?
 \end{question}
 This question is resolved in \cite[Remark 4.6]{Sug24} by using Theorem \ref{1-cohomologically hyperboic}, but the first dynamical degree of the constructed birational map is maximal, but not strictly maximal among the dynamical degrees.\medskip\\
 \noindent
 {\bf Acknowledgements.} The author thanks Professor Keiji Oguiso for several advice on this paper.

\section{Preliminaries}
 There are some properties on dynamical degrees.\par
 \begin{proposition}\label{inverse map}
  Let $X$ be a smooth projective variety over an algebraic closed field of characteristic $0$, of dimension $N$ and let $f:X\dashrightarrow X$ be a birational self-map.
  Then,
  \begin{align*}
   \lambda_p(f)&=\lim_{n\to\infty}((f^n)^{*}(H^{p})\cdot H^{N-p})^{\frac{1}{n}}\\
               &=\lim_{n\to\infty}((H^{p} \cdot (f^n)_{*}(H^{N-p}))^{\frac{1}{n}}\\
               &=\lim_{n\to\infty}((H^{p} \cdot (f^{-n})^{*}(H^{N-p})))^{\frac{1}{n}}\\
               &=\lambda_{N-p}(f^{-1})
  \end{align*}\par
  for $0\leq p\leq N$.
 \end{proposition}
 This result is well-known but we do not know the reference, and the outline of the proof is almost written in \cite[Section 2.1]{Sug24} in the case $\mathrm{dim}(X)=3$.
 We give the proof again at here.
 \begin{proof}[Proof of Proposition \ref{inverse map}]
  The first equation and the fourth equation are derived from the definition of the dynamical degrees.
  The third equiation is the definition of the pushforward and the pullback of cycles.\par
  It remains to confirm the second equation.
  Denote $\pi:X\times X\rightarrow X$.
  Let $\Gamma_f\subset X\times X$ be the closure of the graph of $f$ and let $\pi_1:X\rightarrow \mathrm{Spec}(k)$ be the structure map.
  Take $\alpha\in A_p(X)$ and $\beta\in A_q(X)$ with $p+q=\mathrm{dim}(X)$.
  By definition, $f^{*}\beta:={\pi}_{*}([\Gamma_f]\cdot\pi^{*}\beta)$ and $f_{*}\alpha:={\pi}_{*}([\Gamma_f]\cdot\pi^{*}\alpha)$ (cf.\ \cite[Definition 16.1.2]{Ful98}).
  Then, the intersection of cycles can be calculated as
  \begin{align*}
   {\pi_1}_{*}(f^{*}(\beta)\cdot\alpha)&={\pi_1}_{*}({\pi}_{*}([\Gamma_f]\cdot{\pi}^{*}\beta)\cdot\alpha)\\
                                     &=({\pi_1}_{*}\circ{\pi}_{*})(([\Gamma_f]\cdot{\pi}^{*}\beta)\cdot{\pi}^{*}\alpha)\\
                                     &=({\pi_1}_{*}\circ{\pi}_{*})({\pi}^{*}\beta\cdot([\Gamma_f]\cdot{\pi}^{*}\alpha))\\
                                     &={\pi_1}_{*}(\beta\cdot{\pi}_{*}([\Gamma_f]\cdot{\pi}^{*}\alpha))\\
                                     &={\pi_1}_{*}(\beta\cdot f_{*}(\alpha))
  \end{align*}\medskip 
 \end{proof}
 \begin{proposition}[{cf.\ \cite[Corollaire 7]{DS05}, \cite[Corollary 1.2]{DN11}, \cite[Lemma 3.7]{Tru15}}]\label{birational conjugate}
  Let $X$ and $Y$ be smooth projective varieties of the same dimension $n$ over an algebraic closed field of characteristic $0$.
  Let $f:X\dashrightarrow X$ and $g:Y\dashrightarrow Y$ be dominant rational maps and let $\pi:X\dashrightarrow Y$ be a birational map with $\pi\circ f=g\circ\pi$.
  \begin{figure}[h]
   \includegraphics[scale=0.3]{birational_conjugate.jpg}
   \label{figure}
  \end{figure}\\
  Then, $\lambda_p(f)=\lambda_p(g)$ for all $0\leq p\leq n$.
 \end{proposition}
 By Proposition \ref{birational conjugate}, the dynamical degrees are invariant by a birational conjugate, and therefore the set
 \begin{align*}
  \mathrm{DynBir}_p(X):=\{\lambda_p(f)\mid f:X\dashrightarrow X \text{ is a birational map}\}\subset\mathbb{R}_{\geq1}
 \end{align*}
 is a birationally invariant set for projective varieties for each $0\leq p\leq\mathrm{dim}(X)$.
 \begin{proposition}[{cf.\ \cite[Theorem 1.4]{Tru15}, \cite[Example 3.7]{DN11}}]\label{product formula}
  Let $X=X_1\times X_2$ be the product of two smooth projective varieties over an algebraically closed field $k$ with characteristic $0$, with $\mathrm{dim}(X_i)=n_i$.
  Let $f_1, f_2$ are dominant rational self-map of $X_1, X_2$ and define $f:=f_1\times f_2:X_1\times X_2\dashrightarrow X_1\times X_2$.\par
  Then, 
  \begin{align*}
   \lambda_l(f)=\underset{\substack{p+q=l \\ 0\leq p\leq n_1\\ 0\leq q\leq n_2}}{\max}\lambda_p(f_1)\lambda_q(f_2).
  \end{align*}
 \end{proposition}

\section{Proof of Theorem \ref{Main theorem}}\label{proof of theorem a}
 In this section, we prove Theorem \ref{Main theorem}.
 First, we prove the next proposition.
 \begin{proposition}\label{n=6}
  there exists a biratinal map $\Psi:\mathbb{P}^6\dashrightarrow\mathbb{P}^6$ whose intermediate dynamical degrees are all transcendental.
 \end{proposition}
 \begin{proof}
  Let $\psi:\mathbb{P}^3\dashrightarrow\mathbb{P}^3$ be a birational map, which is presented in Theorem \ref{1-cohomologically hyperboic}.
  Then,
  \begin{align*}
   \lambda_0(\psi)=1, \lambda_1(\psi)=\lambda, \lambda_2(\psi)=a, \lambda_3(\psi)=1,
  \end{align*}
  where $\lambda$ is trandscendental, $a$ is algebraic, and $\lambda>a$.\par
  $\psi^{-1}:\mathbb{P}^3\dashrightarrow\mathbb{P}^3$ is also a birational map, and
  \begin{align*}
   \lambda_0(\psi^{-1})=1, \lambda_1(\psi^{-1})=a, \lambda_2(\psi^{-1})=\lambda, \lambda_3(\psi^{-1})=1,
  \end{align*}
  by Proposition \ref{inverse map}.\par
  Define the birational map $\Psi':=\psi\times\psi^{-1}:\mathbb{P}^3\times\mathbb{P}^3\dashrightarrow\mathbb{P}^3\times\mathbb{P}^3$.
  Then, by Proposition \ref{product formula},
  \begin{align*}
   \lambda_0(\Psi')=1, \lambda_1(\Psi')=\lambda, \lambda_2(\Psi')=a\lambda, \lambda_3(\Psi')=\lambda^2, \lambda_4(\Psi')=a\lambda, \lambda_5(\Psi')=\lambda, \lambda_6(\Psi')=1
  \end{align*}
  from $\lambda>a$.\par
  Define the birational map
  \begin{align*}
  \begin{array}{cccc}
    \Phi:&\mathbb{P}^3\times\mathbb{P}^{3}&\dashrightarrow&\mathbb{P}^6\\
         & ([x_0:x_1:x_2:x_3],\ [y_0:y_1:y_2:y_3]) & \mapsto & [x_0y_0:x_1y_0:x_2y_0:x_3y_0:x_3y_1:x_3y_2:x_3y_3]
  \end{array}
  \end{align*}
  and define $\Psi:=\Phi\circ\Psi'\circ\Phi^{-1}:\mathbb{P}^6\dashrightarrow\mathbb{P}^6$.\par
  $\Psi$ is a birational map and by Proposition \ref{birational conjugate}, the dynamical degrees are same as $\Psi'$.
  Thus,
  \begin{align*}
   \lambda_0(\Psi)=1, \lambda_1(\Psi)=\lambda, \lambda_2(\Psi)=a\lambda, \lambda_3(\Psi)=\lambda^2, \lambda_4(\Psi)=a\lambda, \lambda_5(\Psi)=\lambda, \lambda_6(\Psi)=1.
  \end{align*}
  Now $\lambda$ is a transcendental number and $a$ is an algebraic integer, so the intermediate dynamical degrees of $\Psi$ are all transcendental.
 \end{proof}
 \begin{proof}[Proof of Theorem \ref{Main theorem}]
  We prove this theorem by induction in $n$.\par
  If $n=6$, the theorem is proved by Proposition \ref{n=6}.\par
  Assume the theorem holds for $n=n_0$.
  Let $g:\mathbb{P}^{n_0}\dashrightarrow\mathbb{P}^{n_0}$ be a birational map, whose intermediate dynamical degrees are transcendental.\par
  Denote $\lambda_i(g)=\lambda_i$ ($0\leq i\leq n_0$).
  $g\times\id$ is a birational map and the dynamical degrees of $g\times\id:\mathbb{P}^{n_0}\times\mathbb{P}^1\dashrightarrow\mathbb{P}^{n_0}\times\mathbb{P}^1$ can be calculated by using Proposition \ref{product formula} as
  \begin{align*}
   \lambda_0(g\times\id)&=\lambda_0=1\\
   \lambda_1(g\times\id)&=\mathrm{max}\{\lambda_0,\lambda_1\}=\lambda_1\\
   \lambda_2(g\times\id)&=\mathrm{max}\{\lambda_1,\lambda_2\}\\
   &\ \,\vdots\\
   \lambda_{i}(g\times\id)&=\mathrm{max}\{\lambda_{i-1},\lambda_{i}\}\\
   &\ \,\vdots\\
   \lambda_{n_0-1}(g\times\id)&=\mathrm{max}\{\lambda_{n_0-2},\lambda_{n_0-1}\}\\
   \lambda_{n_0}(g\times\id)&=\mathrm{max}\{\lambda_{n_0-1},\lambda_{n_0}\}=\lambda_{n_0-1}\\
   \lambda_{n_0+1}(g\times\id)&=\lambda_{n_0}=1.
  \end{align*}
  Thus, the intermediate dynamical degrees of $g\times\id$ are all transcendental.
  Define the birational map
  \begin{align*}
  \begin{array}{cccc}
    h:&\mathbb{P}^{n_0}\times\mathbb{P}^{1}&\dashrightarrow&\mathbb{P}^{n_0+1}\\
      & ([x_0:x_1:\cdots:x_{n_0}],\ [y_0:y_1]) & \mapsto & [x_0y_0:x_1y_0:\cdots:x_{n_0}y_0:x_{n_0}y_1]
  \end{array}
  \end{align*}
  and define $g'=h\circ g\circ h^{-1}:\mathbb{P}^{n_0+1}\dashrightarrow\mathbb{P}^{n_0+1}$.
  Now $g'$ is a birational map and by Proposition \ref{birational conjugate}, the dynamical degrees are same as $g\times\id$, and this is the case for $n=n_0+1$.
  Therefore, all intermediate dynamical degrees of $g'$ are transcendental, and the proof is done. 
 \end{proof}
 \begin{remark}
  By using the induction in the proof of Theorem \ref{Main theorem} for $\Psi:\mathbb{P}^6\dashrightarrow\mathbb{P}^6$, we get the birational map $g:\mathbb{P}^d\dashrightarrow\mathbb{P}^d$ ($d\geq6$) with
  \begin{gather*}
   \lambda_0(g)=1, \lambda_1(g)=\lambda, \lambda_2(g)=a\lambda, \lambda_3(g)=\cdots=\lambda_{d-3}(g)=\lambda^2,\\
   \lambda_{d-2}(g)=a\lambda, \lambda_{d-1}(g)=\lambda, \lambda_d(g)=1.
  \end{gather*}
 \end{remark}
 \begin{remark}
  \cite[Section 1]{BDJ21} shows that there exists a dominant rational map $f:\mathbb{P}^2\dashrightarrow\mathbb{P}^2$ such that
  \begin{align*}
   \lambda_0(f)=1, \lambda_1(f)=6.8575574092\cdots, \lambda_2(f)=5
  \end{align*}
  with transcendental $\lambda_1(f)$.
  By applying the induction in the proof of Theorem \ref{Main theorem} for $f$, there exists a dominant rational map $\mathbb{P}^d\dashrightarrow\mathbb{P}^d$ such that its intermediate dynamical degrees are all transcendental, for any $d\geq2$.
 \end{remark}

\renewcommand{\refname}{References}

\end{document}